\RequirePackage{fix-cm}
\documentclass[smallextended]{svjour3}       
\smartqed  
\usepackage{graphicx}
\usepackage{amsmath,amssymb}

\newcommand{\argmin}{\mathop{\rm arg~min}\limits}

\begin{document}

\title{Polynomial algorithm for $k$-partition minimization of monotone submodular function
\thanks{
Relevant with this manuscript, 
2-partition minimization of entropy function and its potential extension to $k$-partition one 
are discussed in our paper under review (https://arxiv.org/abs/1708.01444).
This work was supported by the JSPS KAKENHI Grant-in-Aid for Scientific Research on Innovative Areas JP 16H01609 and  for Scientific Research B (Generative Research Fields) JP 15KT0013. 
}
}


\author{Shohei Hidaka 
}


\institute{S. Hidaka \at
              1-1 Asahidai, Nomi, Ishikawa 923-1292, Japan\\
              Tel.: +81-761-51-1725\\
              \email{shhidaka@jaist.ac.jp}           
}

\date{Received: date / Accepted: date}

\maketitle

\begin{abstract}
For a fixed $k$, this study considers $k$-partition minimization of submodular system $(V, f)$ with a finite set $V$ and symmetric submodular function 
$f: 2^{V} \mapsto \mathbb{R}$. 
Our algorithm uses the Queyranne's (1998) algorithm for 2-partition minimization which arises at each step of the recursive decomposition of subsets of the original $k$-partition minimization. We show that the computational complexity of this minimizer is $O(n^{3(k-1)})$.
%
\keywords{Submodular partition problem \and symmetric submodular function}
\end{abstract}

\section{$k$-partition minimization of submodular system}
Let $(V, f)$ is any submodular system with a finite set $V$ and submodular function  
$f: 2^{V} \mapsto \mathbb{R}$. 
We call function $f: 2^{V} \mapsto \mathbb{R}$ is submodular, if for any $X, Y \subseteq V$ it satisfies
\[
 f( X ) + f( Y ) \ge f( X \cup Y ) + f( X \cap Y ).
\]
We call function $f$ {\it symmetric}, if  $f(U) = f( V \setminus U )$ for any set $U \subseteq V$,
and call it {\it monotone}, if $f( U \cup U' ) \ge f( U ) $ for any set $U, U' \subseteq V$.

In this paper, we consider $k$-partition minimization with monotone submodular function $f$. 
Denote the set of all $k$-partitions for a given set $V$ by 
$$P_{k,V} := \left\{ (U_{0}, U_{1}, \dots, U_{k-1}) | \bigcup_{i}U_{i} = V, U_{i} \cap U_{j} = \emptyset \text{ for any } i \neq j\text{ and, } U_{i} \neq \emptyset \text{ for every } i \right\}.$$ 
$k$-partition minimization problem of a submodular system $(V, f)$ is to find $k$-partition
$U = (U_{0}, U_{1}, \dots, U_{k-1}) \in P_{k, V}$ which minimizes the function 
$$g( U ) := \sum_{i=0}^{k-1} f( U_{i} ) + C,$$
where $C$ is a constant to any choice of $k$-partition.
As a practical application of $k$-partition minimization problem with monotone submodular function, 
Hidaka and Oizumi \cite{HidakaUnderReview} have 
discussed the minimum $k$-partition of the mutual information for subsets with higher integrated information.  
In their study, 
the monotone submodular function is defined by the Shannon entropy of a set of random variables $X$, denoted by $f(X) := H(X)$.
The function $H( X )$ is monotone increasing, as $H( X \cup Y ) - H( X ) = H( Y \mid X )\ge 0$ for any set of variables $Y$.
The $k$-partition function is defined by
$$g( U ) = \sum_{i=0}^{k-1} H( U_{i} ) - H( U ), $$
which is known as total correlation \cite{Watanabe1960} or multi-information \cite{Studeny1999}.


Quyranne \cite{Queyranne1998} has shown a $O( |V|^{3} )$ algorithm for the $2$-partition minization problem. 
Okumoto and colleagues \cite{Okumoto2012} have shown a polynomial algorithm for the $3$-partition minimization problem.
To our knowledge, no study has reported yet a polynomial algorithm for the $k$-partition minimization problem with a fixed $k > 1$ in general.
In this paper, we report a polynomial-time algorithm for $k$-partition minimization of an arbitrary monotone submodular function 
by extending Queyranne's algorithm \cite{Queyranne1998}.

\section{Extension of Queyranne's algorithm}
Queyranne's algorithm \cite{Queyranne1998} works on bi-partition minimization of $$g( U ) = f( U ) + f( V \setminus U )$$ 
for an arbitrary submodular system $(V, f)$ with respect to non-empty set $U \subset V$.

Here we show a recursive algorithm extending it for $k$-partition minimization of $g( U ) = \sum_{i=0}^{k-1}f(U_{i}) $ with respect to 
$U \in P_{k, V}$. 
The basic idea is to reduce the original $k$-partition problem with the objective function $g: P_{k, V} \mapsto \mathbb{R}$ to 
a set function $g_{k, V}: 2^{V} \mapsto \mathbb{R}$ by recursively defining $g_{k-1, U}$ for the remaining $(k-1)$ subsets in a given $k$-partition.

By taking a 3-partition minization as an example, first let us consider the following ``naive'' reduction:
$$
 \min_{U \in P_{3,V}}g( U ) = \min_{\emptyset \subset U_{1} \subset V} 
\left[
f( U_{1} ) + \overbrace{ \min_{ \emptyset \subset U_{2} \subset V \setminus U_{1} } \{ f( U_{2} ) + f( V \setminus ( U_{1} \cup U_{2}) ) \} }^{ g_{2}( V \setminus U_{1} ) }
\right].
$$
The first-level minization is performed on the function $f( U_{1} ) + g_{2}( V \setminus U_{1} )$, where 
$g_{2}( V \setminus M_{1} )$ is defined by the second-level minimization of bi-partition function $f(U_{2}) + f( V \setminus ( U_{1} \cup U_{2}))$. 
In this naive reduction, the second-level minimization is solved by the Queyranne's algorithm, but the first-level function $f( M_{1} ) + g_{2}( V \setminus M_{1} )$
is not symmetric in general. 
In order to let the function at every level be symmetric, let us redefine the reduction as follows: 
$$
 \min_{U \in P_{3,V}}g( U ) = \min_{\emptyset \subset U_{1} \subset V} 
\min
\begin{cases}
f( V \setminus U_{1} )
+
\overbrace{ \min_{ \emptyset \subset U_{2} \subset U_{1} } f_{U_{1}}( U_{2} ) }^{ g_{2}( U_{1} ) }
\\
f( U_{1} ) 
+
\overbrace{ \min_{ \emptyset \subset U_{2} \subset V \setminus U_{1} } 
f_{ V \setminus U_{1} }( U_{2} )
}^{ g_{2}( V \setminus U_{1} ) }
\end{cases},
$$
where $f_{ W }( U ) := f( U ) + f( W \setminus U )$ for any set $U \subseteq W$.
In this formulation, each of the second-level minimization for $g_{2}( U_{1} )$ and $g_{2}( V \setminus U_{1} )$ is solveable with Queyranne's algorithm,
and the first-level function 
$$g_{3, V}(U_{1}) := \min( f( V \setminus U_{1} ) + g_{2}( U_{1} ), f( U_{1} ) + g_{2}( V \setminus U_{1} ) )$$ 
is symmetric.
Thus, the 3-partition minimization of $g(U)$ is solveable by Queyranne's algorithm, if this first-level function $g_{3, V}(U_{1})$ is submodular.
Later, we prove its submodularity.

Before considering with the submodularity of the above function, let us extend the reduction with nested symmetric functions for the $k$-partition minimization as follows. 
We can identify the $k$-partition minimization problem of $g(U)$ to  
\begin{equation}
\label{eq-k-partition}
 \min_{(U_{1}, \ldots, U_{k})\in P_{k, V}} \sum_{i=1}^{k}f(U_{i}) = \min_{ \emptyset \subset U \subset V}g_{k, V}( U ), 
\end{equation}
where
the series of symmetric function is defined for any non-empty subset $U_{1} \subset U_{2} \subseteq V$ by 
$g_{2, U_{2}}( U_{1} ) := f_{U_{2}}(U_{1}) = f(U_{1}) + f( U_{2} \setminus U_{1} )$
and for $k > 2$
\begin{equation}
\label{eq-recursive-def}
g_{k,U_{2}}( U_{1} ) := 
  \min( h_{k-1, U_{2}}( U_{1}), h_{k-1, U_{2}}( U_{2} \setminus U_{1} ) ), 
\end{equation}
where 
\begin{equation}
 h_{k, U_{2}}( U_{1}) = 
\begin{cases}
 f( U_{1} ) +  \min_{ \emptyset \subset U' \subset U_{2} \setminus U_{1} }g_{k, U_{2} \setminus U_{1}}( U' ) & \text{if } |U_{2} \setminus U_{1}| > 1
\\
 \infty & {\text{otherwise}}
\end{cases}
\end{equation}
for any $ k > 2$ and $U_{1} \subset U_{2} \subseteq V$.
For $k=2$, $g_{2, V}( U ) = f_{ V }( U )$, and 
minimization of $f_{V}(U)$ over the set of bi-partitions of $V$ can be computed by Queyranne's algorithm.

\section{Main results}

If the $k^{\text{th}}$ order function $g_{k, V}$ is submodular at every step above, we can apply Queyranne's algorithm to this function at every recursive step.
As $g_{k, V}$ is symmetric by definition, our main question is whether it is submodular. 
The main result, Theorem \ref{thm-k-partition}, states the function $g_{k, V}$ is submodular, if $f$ is monotone submodular.
To prove Theorem \ref{thm-k-partition}, we have the following steps.
\begin{enumerate}
 \item {Lemma \ref{lem-partition-submodularity} shows submodularity of the function $h_{2, V}$.}
 \item {Lemma \ref{lem-minimum} shows the minimum of two submodular functions $\min( f(X), g(X) )$ 
       with monotone difference is submodular.}
 \item {Theorem \ref{thm-k-partition} shows the symmetrized minimum of $k$-partition $g_{k,V}$ is submodular.}
\end{enumerate}





\begin{lemma}[submodularity of minimum bi-partition] \label{lem-partition-submodularity}
For an arbitrary submodular system $(V, f)$ such that the function $f$ is monotone and $f(\emptyset) = 0$. 
The minimum of bi-partition function $$g(X) := 
\begin{cases}
f( V \setminus X ) + \min_{\emptyset \subset A \subset X }f( A ) + f( X \setminus A ) & \text{if }|X| \ge 2
\\
f( V \setminus X ) + f( X ) & \text{otherwise}
\end{cases}
$$ is submodular.
\end{lemma}
\begin{proof}
If $|X| < 2$, $g$ is obviously submodular, and thus suppose $|X| \ge 2$.
For $\emptyset \subset W \subset Z \subseteq V$, write $f_{Z}(W):= f(W) + f(Z \setminus W)$.
Denote one of the minimal sets for the following functions by 
\begin{eqnarray}
\nonumber
A_{1} := \argmin_{ \emptyset \subset A' \subset X }f_{ X }( A' ), & B_{1} := \argmin_{ \emptyset \subset B' \subset Y} f_{ Y }( B' ), 
\\
\nonumber
W_{1} := \argmin_{ \emptyset \subset W' \subset X \cup Y}f_{ X \cup Y }( W' ), & Z_{1} := \argmin_{ \emptyset \subset Z' \subset X \cap Y} f_{ X \cap Y }( Z' ),
\end{eqnarray} 
and their another subset of bi-partition by 
$$A_{2} = X \setminus A_{1}, \ B_{2} = V \setminus B_{1}, \ W_{2} = ( X \cup Y ) \setminus W_{1}, \text{ and } Z_{2} = ( X \cap Y ) \setminus Z_{1},$$
and their complements by 
$$A_{3} = V \setminus X, \ B_{3} = V \setminus Y, \ W_{3} = V \setminus ( X \cup Y ), \text{ and } Z_{3} = V \setminus ( X \cap Y ).$$

If $X \cap Y = \emptyset$, $f( Z_{1}) = f( Z_{2} ) = f(A_{i} \cap B_{j}) = 0$ for any $i, j = 1, 2$, and by the minimality of $f(Z_{1}) + f(Z_{2})$ and $f(W_{1}) + f(W_{2})$, we have
$$ f( A_{1} \cup B_{1} ) + f( A_{1} \cap B_{1} ) + f( A_{2} \cap B_{2} ) + f( A_{2} \cup B_{2} ) \ge f( W_{1}) + f(W_{2}) + f( Z_{1} ) + f(Z_{2}).$$
By the submodular inequality, 
\begin{equation}
\label{eq-12}
f( A_{1} ) + f( B_{1} ) + f( A_{2} ) + f( B_{2} ) \ge f( W_{1}) + f(W_{2}) + f( Z_{1} ) + f(Z_{2})
\end{equation}
and 
\begin{equation}
\label{eq-3}
f( A_{3} ) + f( B_{3} ) \ge f( W_{3}) + f(Z_{3}). 
\end{equation}
 Adding these inequalities, $g(X)$ holds the submodular inequality.

Concider the second case that holds $X \cap Y \neq \emptyset$, $A_{i} \cap B_{j} \neq \emptyset$. 
Then, by the minimality of $f(Z_{1})+f(Z_{2})$ and $f(W_{1}) + f(W_{2})$, 
we have 
$$f( A_{i} \cap B_{j} ) + f( ( A_{3 - i} \cup B_{3 - j} ) \cap X \cap Y ) \ge f( Z_{1} ) + f( Z_{2} ),$$
and
$$f( A_{i} \cup B_{j} ) + f( ( A_{3 - i} \cap B_{3 - j} ) ) \ge f( W_{1} ) + f( W_{2} ).$$
By monotonicity of $f$, $f( ( A_{3 - i} \cup B_{3 - j} ) ) \ge f( ( A_{3 - i} \cup B_{3 - j} ) \cap X \cap Y )$, and by the submodularity inequality
we have (\ref{eq-12}).
Adding (\ref{eq-3}) to (\ref{eq-12}), $g(X)$ holds the submodular inequality.

\end{proof}

Lemma \ref{lem-partition-submodularity} states the minimum bi-partition function is submodular, if $f$ is monotone submodular function.
But note that this function is not symmetric as it is, and slightly different from the function (\ref{eq-recursive-def}), that we defined earlier so it can be minimized by Queyranne's algorithm.
As the function (\ref{eq-recursive-def}) takes additional minimum to be symmetric, 
we need to deal with this minimum of two submodular functions by showing the following Lemma \ref{lem-minimum}.

\begin{lemma}[submodularity of minimum of two submodular functions]
\label{lem-minimum}
 For two submodular function $f$ and $g$ over the ground set $V$, 
 \[
  h( X ) = \min( f( X ), g(X) )
 \]
is submodular, if the function $d( X ) := f( X ) - g( X )$ is either monotone increasing or decreasing.
\end{lemma}
\begin{proof}
If $h( X ) + h( Y ) = f( X ) + f( Y )$ or $h( X ) + h( Y ) = g( X ) + g( Y )$, by submodularity we have
\[
 h( X ) + h( Y ) \ge \min( f( X \cup Y ), g( X \cup Y ) ) + \min( f( X \cap Y ), g( X \cap Y ) ) = h( X \cup Y ) + h( X \cap Y ).
\]
Otherwise, $h( X ) + h( Y ) = f( X ) + g( Y )$ or $h( X ) + h( Y ) = g( X ) + f( Y )$.
As $d( X )$ is monotone, it holds either 
$$
f( X ) \ge f( X \cup Y ) - g( X \cup Y ) + g( X )  \text{ or } g( Y ) \ge f( Y ) - f( X \cup Y ) + g( X \cup Y ).
$$
By submodularity of $f$ and $g$ 
$$
f( X ) + g( Y ) \ge f( X \cup Y ) + g( X \cap Y )  \text{ or } f( X ) + g( Y ) \ge  g( X \cup Y ) + f( X \cap Y ).
$$
Similarly, 
$$
g( X ) + f( Y ) \ge g( X \cup Y ) + f( X \cap Y )  \text{ or } g( X ) + f( Y ) \ge  f( X \cup Y ) + g( X \cap Y ).
$$
Thus, 
$$ 
h( X ) + h( Y ) \ge h( X \cup Y ) + h( X \cap Y ).
$$
\end{proof}

Combining Lemma \ref{lem-partition-submodularity} and Lemma \ref{lem-minimum}, 
the following theorem states submodularity of the symmetrized minimum $k$-partition.

\begin{theorem}[submodularity of symmetric minimum $k$-partition]
\label{thm-k-partition}
For $k > 1$ and an arbitrary submodular system $(V, f)$ with monotone submodular function $f$, 
for $X \subseteq V$
define
$g_{2, X}( Y ) := f( Y ) + f( X \setminus Y )$, and for $k > 2$
$$
g_{k, X}( Y ) = \min\left( h_{k-1,X}( Y ), h_{k-1,X}( X \setminus Y ) \right)
$$
and
$$
 h_{k, X}( Y ) := f( Y ) + g_{k}( X \setminus Y ). 
$$ 
The set function $g_{k,V}: 2^{V} \mapsto \mathbb{R}$
is submodular for any $k \ge 2$.
\end{theorem}
\begin{proof}
Any function $h_{2,X}(Y)$ for $\emptyset \subset Y \subset X \subset V$ is submodular due to Lemma \ref{lem-partition-submodularity}.
By induction, suppose that $h_{m, X}(Y)$ is submodular for $k = 2, \ldots, m$, and let us show $g_{m+1,X}(Y)$ is submodular.
By Lemma \ref{lem-minimum}, it is sufficient to show $$d_{m, X}( Y ) := h_{m, X}( Y ) - h_{m, X}( X \setminus Y )$$ is either monotone decreasing or monotone increasing.
For any singleton set $S \subseteq V$ and $|S| = 1$, 
\begin{eqnarray}
\nonumber
 d_{m,X}( Y ) &-& d_{m,X}( Y \cup S ) 
\\
\nonumber
&=& 
f( Y ) + \min_{U \in P_{ k, X \setminus Y } }\sum_{i=1}^{k}f(U_{i})
- 
f( X \setminus Y ) - \min_{U \in P_{ k, X } }\sum_{i=1}^{k}f(U_{i})
\\
\nonumber
-
f( Y \cup U ) &-& \min_{U \in P_{ k, X \setminus ( Y \cup U ) } }\sum_{i=1}^{k}f(U_{i})
+ 
f( X \setminus ( Y \cup U ) ) + \min_{U \in P_{ k, Y \cup U } }\sum_{i=1}^{k}f(U_{i})
\end{eqnarray}
Write the minimal $k$-partitions
$$
( W_{1}, \ldots, W_{k} ) := \argmin_{ ( U_{1}, \ldots, U_{k} ) \in P_{k, X \setminus Y}} \sum_{i=1}^{k}f( U_{i} )
$$
and
$$
( Z_{1}, \ldots, Z_{k} ) := \argmin_{ ( U_{1}, \ldots, U_{k} ) \in P_{k, Y + U }} \sum_{i=1}^{k}f( U_{i} )
.$$
We have the following inequalities 
\[
 \sum_{i=1}^{k}f(W_{i}) 
- \min_{ U \in P_{k, V \setminus (X + U)}} \sum_{i=1}^{k}f(U_{i})
\ge 
 \sum_{i=1}^{k}\delta( S \subseteq W_{i} )\left( f(W_{i}) - f( W_{i} \setminus S ) \right)
\]
and
\[
  \sum_{i=1}^{k}f(Z_{i}) 
- \min_{ U \in P_{k, X}} \sum_{i=1}^{k}f(U_{i})
\ge 
 \delta( S \subseteq Z_{i} )( f( Z_{i} ) - f( Z_{i} \setminus S ) )
,
\]
where $\delta( P ) = 1$ if the statement $P$ is true, and $\delta( P ) = 0$ otherwise.
By submodularity of $f$, we have the following inequalities for any $i$
$$
f( Y \cup S ) - f(Y) \le f( Z_{i} ) - f( Z_{i} \setminus S )
\text{ and } 
f(X \setminus Y) - f(X \setminus (Y \cup S)) \le 
f( W_{i} ) - f( W_{i} \setminus U ).
$$
Inserting these inequalities, we have
\[
  d_{m,X}( Y ) - d_{m,X}( Y + S ) 
\ge
0,
\]
and it implies $d_{m,X}( Y )$ is monotone decreasing.
\end{proof}


Theorem \ref{thm-k-partition} states that the $k^{\text{th}}$ order function $g_{k, V}$ is submodular at every step above, 
and thus we can apply Queyranne's algorithm to this function at every recursive step.
As minimization of a $k$-partition function includes minimization of the $(k-1)$-partition function, 
the number of required times to call the function $f$ is $O(n^{3(k-1)})$ for 
this recursive algorithm for the minimal $k$-partition of $n$ elements.


%
%

\begin{acknowledgements}
We thank Masafumi Oizumi for his encouragement to complete this manuscript. 
\end{acknowledgements}


\bibliographystyle{spmpsci}
\bibliography{TotalCorrReferences}

\begin{thebibliography}{1}
\providecommand{\url}[1]{{#1}}
\providecommand{\urlprefix}{URL }
\expandafter\ifx\csname urlstyle\endcsname\relax
  \providecommand{\doi}[1]{DOI~\discretionary{}{}{}#1}\else
  \providecommand{\doi}{DOI~\discretionary{}{}{}\begingroup
  \urlstyle{rm}\Url}\fi

\bibitem{HidakaUnderReview}
Hidaka, S., Oizumi, M.: Fast and exact search for the partition with minimal
  information loss.
\newblock PLoS ONE  (under review)

\bibitem{Okumoto2012}
Okumoto, K., Fukunaga, T., Nagamochi, H.: Divide-and-conquer algorithms for
  partitioning hypergraphs and submodular systems.
\newblock Algorithmica \textbf{62}(3), 787--806 (2012)

\bibitem{Queyranne1998}
Queyranne, M.: Minimizing symmetric submodular functions.
\newblock Mathematical Programming \textbf{82}(1-2), 3--12 (1998)

\bibitem{Studeny1999}
Studen\'y M~\&~Vejnarov\'a, J.: The multiinformation function as a tool for
  measuring stochastic dependence.
\newblock MIT Press, Cambridge, MA (1999)

\bibitem{Watanabe1960}
Watanabe, S.: Information theoretical analysis of multivariate correlation.
\newblock IBM J. Res. Dev. \textbf{4}(1), 66--82 (1960).
\newblock \doi{10.1147/rd.41.0066}.
\newblock \urlprefix\url{http://dx.doi.org/10.1147/rd.41.0066}

\end{thebibliography}

%
%

\end{document}